\documentclass{amsart}
\usepackage{amscd,amssymb,hyperref}
\usepackage[all]{xy}
\theoremstyle{plain}
\newtheorem{prop}[subsection]{Proposition}
\newtheorem{thm}[subsection]{Theorem}
\newtheorem{lem}[subsection]{Lemma}
\newtheorem{cor}[subsection]{Corollary}

\theoremstyle{remark}

\theoremstyle{definition}

\numberwithin{equation}{section}

\title{On ``small geodesics" and free loop spaces}

\author[A.~Bahri]{A.~Bahri}
\address{Department of Mathematics, Rutgers University,
New Brunswick, NJ}
\email{\href{mailto:abahri@math.rutgers.edu}{abahri@math.rutgers.edu}}
\thanks{}
\author[F. R.~Cohen]{F. R.~Cohen}
\address{Department of Mathematics, University of Rochester,
Rochester, NY 14627}
\email{\href{mailto:cohf@math.rochester.edu}{cohf@math.rochester.edu}}
\subjclass{Primary: 55R99, Secondary: 58D99}

\begin{document}
\begin{abstract}
A topological group is constructed which is homotopy equivalent to
the pointed loop space of a path-connected Riemannian manifold $M$
and which is given in terms of ``composable small geodesics" on $M$.
This model is analogous to J.~Milnor's free group construction
\cite{Milnor} which provides a model for the pointed loop space of a
connected simplicial complex. Related function spaces are
constructed from ``composable small geodesics" which provide models
for the free loop space of $M$ as well as the space of continuous
maps from a surface to $M$.
\end{abstract}
\maketitle

\section{Introduction}

The main purpose of this article is to describe a ``combinatorial
model" which is a topological group that is homotopy equivalent to
the pointed loop space of a Riemannian manifold $M$. A second
purpose is to give a model for the free loop space $\Lambda M$, the
space of all continuous maps from the circle to $M$. As a
consequence, a model for the space of continuous maps of a closed
orientable surface to $M$ is also given.

The pointed loop space $\Omega M$ is useful here. That is the
subspace of $\Lambda M$ given by functions which preserve a point,
namely $f(*) = *_M$ for fixed points $*$ in the circle, and $*_M$ in
$M$. The method used here is implicit in a method due to J.~Milnor
\cite{Milnor}. Thus this article is partially an exposition of those
methods, but in a somewhat different as well as extended context.
The main point is the construction of a topological group arising
from ``small geodesics" which has the homotopy type of $\Omega M$.

Combinatorial models are obtained by patching together ``composable
small geodesics" on $M$ as follows. The model is given by the space
of all ordered $n$-tuples of points where successive points are
required (1) to differ and (2) to satisfy the additional property
that there is an unique minimal geodesic between successive points.
With mild hypotheses, the space of all such $n$-tuples for $ n \geq
0$ are assembled into a topological group which is weakly homotopy
equivalent to the pointed loop space $\Omega M$. Forming the
homotopy orbit space for this group acting on itself by the adjoint
representation then gives a model for the free loop space $\Lambda
M$.

One further application is a combinatorial model for the space of
pointed maps as well as free maps of a closed, orientable Riemann
surface to $M$ which is given in terms of ``composable small
geodesics".

These models arise by introducing a natural monoid structure on
``composable minimal geodesics" as described explicitly below as
well as exploiting the structure of the radius of convexity for a
Riemannian manifold. One feature of one combinatorial model here is
that it is naturally a topological group while the multiplication in
the loop space $\Omega M$ is only associative up to homotopy.

Combinatorial models have been useful for $50$ years. Some examples
were first given in work of I.~M. James \cite {James}, and J.~Milnor
\cite{Milnor}. Later models for $\Omega^n\Sigma^n(X)$ were given by
J.~Milgram and J.~P.~May \cite{May,Milgram}. Combinatorial models
for the free loop space of a suspension are given in
\cite{Cfb,Ralph} while models for the space of continuous maps of
the $n$-sphere to an $(n+1)$-fold suspension are given in
\cite{bcp}.

Throughout this article $M$ is assumed to be a Riemannian manifold.
Some definitions as well as proofs are patterned quite closely on an
early result of Milnor \cite{Milnor} which gave combinatorial models
for loop spaces of connected simplicial complexes for which
modifications are made here for ``small geodesics". Related models
for $\Lambda M $ given by ``small geodesics" are also given in
\cite{Abbas}. This article assembles some of that structure in a
language familiar to topologists with the main modification required
here given by Lemma \ref{lem:convex balls}.

The authors thank Ryan Budney, and Mike Gage for explaining the
convexity radius of a Riemannian manifold. The first author is
grateful for the hospitality from the Institute for Advanced Study
during some of the work on this article. Much of this article dates
back to conversations of the authors in $2002$.

\section{ Constructions, and statement of results}

The purpose of this section is to assemble a space built out of
``small geodesics" as well as recording basic properties of these
constructions. Before giving these constructions, recall that a
Riemannian manifold $M$ admits a cover by geodesic balls (for
example page $1123$ section $360$C of \cite{EDM}). This feature is
used in the first definition as follows.

\begin{enumerate}
\item Let $$Z(M,k)$$ denote the subspace of $M^{k+1}$ given by
$(k+1)$-tuples $(x_k,x_{k-1}, \cdots , x_0)$ which satisfy the
property that there is an unique minimal geodesic from $x_i$ to
$x_{i+1}$ for all  $0 \leq i < k$.

\item Fix a point $v_0$ in $M$. Let $Z(M,k,v_0)$ denote the subspace of $Z(M,k)$
with $x_0 =  v_0$.

\item Let $X(M,k)$ denote the subspace of $Z(M,k)$ with $x_0 =
x_k$.

\item Let $G(M,k)$ is the subspace of $X(M,k)$ with $x_0 = x_k  = v_0$.
\end{enumerate}

Consider the disjoint union $\amalg_{k \geq 0}Z(M,k)$. Let
$$Z(M,\infty)$$ denote the identification space obtained from
the equivalence relation on $\amalg_{k \geq 0}Z(M,k)$ generated by
$$(x_k,x_{k-1}, \cdots, x_i, \cdots, x_0) \sim (x_k,x_{k-1}, \cdots,
\hat x_i, \cdots, x_0)$$ whenever $x_i = x_{i+1}$ or $x_{i+1} =
x_{i-1}$. Notice that this equivalence relation is exactly that
given in \cite{Milnor}, page $274$,  where a topological group is
constructed which is homotopy equivalent to the loop space of a
connected simplicial complex.

Let $$[x_k,x_{k-1}, \cdots , x_0] $$ denote the associated
equivalence class of $(x_k,x_{k-1}, \cdots , x_0)$ in $Z(M,\infty)$.
Similarly, let
\begin{enumerate}
\item $X(M,\infty)$ be the subspace of $Z(M,\infty)$ given by the
image of $\amalg_{k \geq 0}X(M,k)$ in $Z(M,\infty)$,
\item $Z(M,\infty,v_0)$ be the subspace of $Z(M,\infty)$ given by the
image of $\amalg_{k \geq 0}Z(M,k,v_0)$ in $Z(M,\infty)$, and
\item $G(M,\infty)$ be the subspace of $Z(M,\infty)$ given by the
image of $\amalg_{k \geq 0}G(M,k)$ in $Z(M,\infty)$.
\end{enumerate}

Next notice that the first coordinate projection map $\pi_k:Z(M,k)
\to\ M$ given by $\pi_k((x_k,x_{k-1}, \cdots, x_i, \cdots,
x_0))=x_k$ is continuous. Since the projection maps $\pi_k$ preserve
the equivalence relation for $\amalg_{k \geq 0}Z(M,k)$, there is an
induced continuous map $$\pi:Z(M,\infty) \to\ M.$$

Two useful, technical lemmas are stated next. Proofs are analogous
to that of Lemma \ref{lem:contractability} in section 4 and are
omitted.
\begin{lem}\label{lem:path-connectivity}
If $M$ is a path-connected Riemannian manifold, then
\begin{itemize}
  \item $Z(M,\infty,v_0)$,
  \item  $X(M,\infty)$, and
  \item
 $G(M,\infty)$
\end{itemize} are path-connected. Furthermore, $\pi$ is a surjection.
\end{lem}

In addition, there is a ``partial product" motivated by the
fundamental groupoid $$\mu:Z(M,j,v_0) \times G(M,k) \to\
Z(M,j+k,v_0)$$ given by the formula
$$\mu((x_j,\cdots, x_0),(y_{k},\cdots, y_0)) = (x_j,\cdots,
x_0,y_{k},\cdots, y_0)$$ as $x_0 = y_k = y_0 = v_0$. The map $\mu$
is continuous as it is induced by the inclusion of a subspace (
which has the subspace topology ).

\begin{lem} \label{lem:continuity}
There are continuous maps $\mu:Z(M,\infty,v_0) \times G(M,\infty)
\to\ Z(M,\infty,v_0)$ together with the following commutative
diagram.
\[
\begin{CD}
G(M,\infty) \times G(M,\infty) @>{\mu}>> G(M,\infty)\\
 @VVV           @VV{}V     \\
Z(M,\infty,v_0) \times G(M,\infty) @>{\mu}>> Z(M,\infty,v_0)
\end{CD}
\] In addition $G(M,\infty)$ is a topological group
with identity given by the equivalence class $[v_0]$, and
$G(M,\infty)$ acts on $Z(M,\infty,v_0)$ via the map $\mu$.
\end{lem}

This action has further properties as proven below.
\begin{lem} \label{lem:universal.bundle}

The natural orbit space obtained from the right-action of
$$\mu:Z(M,\infty,v_0) \times G(M,\infty) \to\ Z(M,\infty,v_0)$$ gives
the projection $$p:Z(M,\infty,v_0) \to Z(M,\infty,v_0)/
G(M,\infty)$$ which is the projection in a fibre bundle.
\end{lem}

Let $P(M)$ denote the pointed path-space of $M$, that is the based
continuous functions $f:[0,1] \to\ M$ such that $f(0)= v_0$.

\begin{thm}\label{thm:bundle}
Assume that $M$ is a path-connected Riemannian manifold.
\begin{enumerate}
  \item The space $Z(M,\infty,v_0)$ is contractible.
  \item The projection $\pi: Z(M,\infty,v_0) \to\ M$ is the
projection in a principle $G(M,\infty)$-bundle.
\item  There is a map $\Theta: X(M,\infty)\to\ \Lambda(M)$
together with

morphisms of fibrations
\[
\begin{CD}
 G(M,\infty) @>{\Theta}>> \Omega(M)   \\
 @VVV           @VV{}V     \\
Z(M,\infty,v_0)  @>{\Theta}>> P(M) \\
 @VV{\pi}V           @VV{}V     \\
M @>{1}>> M
\end{CD}
\] and

\[
\begin{CD}
 G(M,\infty) @>{\Theta}>> \Omega(M)   \\
 @VVV           @VV{}V     \\
X(M,\infty)  @>{\Theta}>> \Lambda(M) \\
 @VV{\pi}V           @VV{}V     \\
M @>{1}>>             M.
\end{CD}
\]

\item The maps $$\Theta: G(M,\infty) \to\ \Omega(M),$$ and
$$\Theta:X(M,\infty)\to\ \Lambda(M)$$
are weak homotopy equivalences.
\end{enumerate}

\end{thm}

The proof of this theorem is given in sections $3$, $4$ and $5$.

The analogous case for certain topological groups is given next
where the free loop space of $G/\Gamma$ is considered with $G$ a
simply-connected topological group and with $\Gamma$ a closed
discrete subgroup of $G$. The structure of $\Lambda(G/\Gamma)$ for
some of these are then tied in below with the model above using
``small geodesics".

First, let $\Gamma$ denote a discrete group which acts freely and
properly discontinuously on a manifold $M$ ( via a left action ).
Let $$p: M \to\ M/\Gamma $$ denote the natural associated covering
space projection. The structure of $\Lambda(M/ \Gamma)$ is standard,
follows from properties of covering spaces and is described next for
purposes of exposition.

A point in the orbit space $M/\Gamma$ is the orbit
$$[m] = m \cdot \Gamma$$ for $m$ in $M$. Let $g$ denote an element of $\Gamma$ and
define $$P_g(M)$$ to be the paths in $M$ which start at a fixed
point $*_M$ which end at $g(*_M)$. Define $$\Theta: \amalg_{g
\epsilon \Gamma} P_g(M) \to\ \Omega(M) $$ by the formula
$\Theta(f)(t) = p(f(t))$.

Next, assume that $M = G$ is a simply-connected topological group
with $\Gamma$ a discrete subgroup of $G$. In this case, there is a
natural left $G$-action on $\amalg_{g \epsilon \Gamma} P_g(G)$ given
by conjugation specified by the formula
$$[\alpha(f)](t) = \alpha \cdot (f(t)) \cdot \alpha^{-1}$$ for $f$ in
$\amalg_{g \epsilon \Pi}P_g(G)$ $\alpha$ in $\Pi$. Furthermore, let
$1_G$ denotes the identity element in the topological group $G$
which is assumed to be the base-point. Evidently, conjugation by any
element in $G$ preserves $1_G$. Using the action of $G$ on itself
via right multiplication by the inverse of an element $\alpha$ in
$G$, there is an induced diagonal action of $G$ on $$ G \times
\amalg_{g \epsilon \Gamma} P_g(G)$$ together with an induced map
$$\Phi: G \times_{\Pi} \amalg_{g \epsilon \Gamma} P_g(G) \to\
\Lambda(G/\Gamma)$$ given by
$$\Phi(m,f)(t) = \pi[m \cdot f(t)].$$ Notice that there is a commutative
diagram

\[
\begin{CD}
\amalg_{g \epsilon \Gamma} P_g(G) @>{\Phi}>>  \Omega(G/\Gamma)  \\
 @VVV           @VV{}V     \\
G \times_{\Gamma} \amalg_{g \epsilon \Gamma} P_g(G) @>{\Phi}>> \Lambda(G/\Gamma) \\
 @VVV           @VV{}V     \\
G/\Gamma @>{1}>> G/\Gamma.
\end{CD}
\]

The next proposition is standard.
\begin{prop}\label{prop:adjoint.rep}
\begin{enumerate}
  \item If $M$ is a path-connected space, then
$$\Theta: \amalg_{g \epsilon \Pi} P_g(M) \to\ \Omega(M)$$
is a homeomorphism.
  \item If $G$ is a simply-connected Lie group and $\Pi$ is a discrete
subgroup of $G$, then $$\Phi: G \times_{\Pi} \amalg_{g \epsilon \Pi}
P_g(G) \to\ \Lambda(G/\Pi)$$ is a homeomorphism.
\end{enumerate}
\end{prop}

A combinatorial model for the spaces of maps of a surface $S_g$ to
$M$ is given next. Fix a topological group $G$ and consider the
commutator map $$\chi_g: G^{2g} \to G$$ given by
$$\chi_g(x_1,x_2, \cdots,x_{2g}) = [x_1,x_2]\cdot [x_3,x_4] \cdots
[x_{2g-1},x_{2g}]$$ for which $[a,b] = aba^{-1}b^{-1}$.

Let $$\xi(\chi_{g})$$ denote the homotopy theoretic fibre of the map
$\chi_g$. Then it is classical that there is a map
$$\alpha: \xi(\chi_{g}) \to map_*(S_g,BG)$$ which is a homotopy equivalence.
In case $G =  G(M,\infty)$, then $BG$ is homotopy equivalent to $M$
by Theorem \ref{thm:bundle}.

A specific concrete model for this homotopy theoretic fibre is
described next. Let $PG$ denote the path-space for $G$, that is the
space of continuous functions $f\colon [0,1] \to G$ such that $f(0)
= 1_G$ ( Any choice of base-point suffices here. )

Thus there is a projection $$p\colon PG \to G$$ given by $$p(f) =
f(1).$$ Define $\xi_g(G)$ as the pull-back in the following diagram:

\[
\begin{CD}
\xi_g(G) @>>> PG \\
@VV{}V      @VV{p}V  \\
G^{2g} (Y)@>>{\chi_g}> G
\end{CD}
\]

One example is given by $$G = G(M,\infty)$$ the topological group
given above which is homotopy equivalent $\Omega(M)$. This group
will be used next to give a model for the space of pointed
continuous maps $map_*(S_g,M)$.

The final example is the space of maps $map(S^1,BG)$ where it is
again classical that this space is homotopy equivalent to the
homotopy orbit space ( sometimes called the Borel construction )
$$EG \times_G G^{ad}$$ where $G^{ad}$ denotes the left $G$-space
where the action is conjugation. Again, let $$G = G(M,\infty)$$ and
$$EG = Z(M,\infty,v_0).$$

\begin{cor}\label{thm:maps.from.surfaces.to.BG}
Assume that $M$ is a simply-connected Riemannian manifold.
\begin{enumerate}
  \item If  $G = G(M,\infty)$, then there is a weak homotopy equivalence
$$\xi_g(G) \to map_*(S_g,M).$$
\item If  $G = G(M,\infty)$ and $EG = Z(M,\infty,v_0)$, then there is a weak homotopy equivalence
$$EG \times_G \xi_g(G) \to map(S_g, M).$$
  \item If  $G = G(M,\infty)$ and $EG = Z(M,\infty,v_0)$, then there is a weak homotopy equivalence
  $$EG \times_G G^{ad} \to \Lambda M.$$
\end{enumerate}

\end{cor}

This article is organized as follows.
\begin{description}
\item[1] Introduction
\item[2] Constructions statements of results
\item[3] Proof of Theorem \ref{thm:bundle}
\item[4] Contractibility of $Z(M,\infty,v_0)$
\item[5] Maps to free loop spaces
\item[6] Two problems
\item[7] Acknowledgements
\end{description}

\section{Proof of Theorem \ref{thm:bundle}}

Additional information required here is listed next. Let $$G(M)$$
denote the space of ordered pairs $(a,b)$ in $ M \times M$ such that
both $a$ and $b$ lie in some convex ball $W$. Thus there is an
unique geodesic arc $f_{(a,b)}(-)$ from $a$ to $b$ which is
parameterized by arc-length with image in $W$. There is an
associated function $\Phi:G(M) \to\ map([0,1],M)$ given by
$\Phi((a,b)) = f_{(a,b)}(-)$.

\begin{lem}\label{lem:convex balls}
Assume that $M$ is a path-connected Riemannian manifold.
\begin{enumerate}
  \item Let $p$ denote a point of $M$. Then there exists an $\epsilon > 0$ such
that there is a convex ball of radius $\epsilon$ containing $p$.
  \item Furthermore, the function $\Phi:G(M) \to\ map([0,1],M)$
is continuous.
\end{enumerate}
\end{lem}

The first part of the proof is given in \cite{EDM}, page $1123$
while the second follows directly from properties of the
compact-open topology for the function space $map([0,1],M)$. Details
are included in section $4$ below.

\begin{lem}\label{lem:principal G bundle}
The projection $\pi: Z(M,\infty,v_0) \to\ M$ is the projection in a
principal $G(M,\infty)$-bundle.
\end{lem}

\begin{proof}
Fix a point $p$ in $M$. By \ref{lem:convex balls}, it may be assumed
that there exists a neighborhood $U_p$ of $p$ for which there is a
unique minimal geodesic between any two points $x$ and $y$ in $U_p$.
Thus the pair $(x,p)$ is in $Z(M,2)$ for $x$ in $U_p$.

Notice that by \ref{lem:path-connectivity}, $\pi^{-1}(p)$ is
non-empty. Let $e_p$ denote a fixed choice of point in $\pi^{-1}(p)$
regarded as a subspace of $Z(M,\infty,v_0)$. Hence
$$e_p = [p,x_{k-1}, \cdots, x_1,v_0].$$ Write
$$e_p^{-1} = [v_0,x_{1},x_2, \cdots, x_{k-1},p],$$ a point in
$Z(M,\infty)$, but not in $Z(M,\infty,v_0)$.

Let $g$ denote a point in $G(M,\infty)$, then $g = [v_0,y_{n-1},
\cdots, y_1,v_0]$. Thus by the definition of the map $\mu$, the
element $[x,p]\cdot e_p \cdot g$ is an element of $Z(M,\infty,v_0)$.
Define $\phi_p: U_p \times G(M,\infty) \to\ Z(M,\infty,v_0)$ by the
formula $$\phi_p(x,g)= [x,p]\cdot e_p \cdot g.$$ By
\ref{lem:continuity}, $\phi_p$ is continuous. Furthermore,
$$\pi(\phi_p(x,g)) = x $$ by definition. Thus $\phi_p$ takes
values in $\pi^{-1}(U_p)$ and is a continuous function
$$\phi_p: U_p \times G(M,\infty) \to\ \pi^{-1}(U_p).$$

Let $e$ be any point in $\pi^{-1}(U_p)$. Since $e = [x_k, \cdots,
x_1, v_0 ]$ $\pi(e)= x_k$ in $U_p$, there is an unique minimal
geodesic from $p$ to $\pi(e)$ $[p,x_k] = [p,\pi(e)]$. Notice that
the product $$e_p^{-1} \cdot [p,\pi(e)]\cdot e$$ is equal to
$$[v_0,x_1, \cdots, x_{k-1},p] \cdot [p,x_k]\cdot [x_k,x_{k-1}
\cdots, v_0 ].$$ Hence this product is equivalent to
$$[v_0,x_1, \cdots, x_{k-1},p ,x_k, x_{k-1}, \cdots, v_0],$$ and
is an element of $G(M,\infty)$.

Define $$\theta_p:\pi^{-1}(U_p) \to\ G(M,\infty) $$ by the formula
$\theta_p(e) = e_p^{-1} \cdot [p,\pi(e)] \cdot e$. By the previous
paragraph, this function is continuous. There is an associated
continuous function $$\Theta:\pi^{-1}(U_p) \to\ U_p \times
G(M,\infty) $$ given by $$\Theta(e) =  (\pi(e), \theta_p(e)).$$

The following formulas are satisfied which imply directly that both
$\Theta$ $\phi_p$ are homeomorphisms so the projection map is
locally trivial.

\begin{enumerate}
  \item $(\phi_p)(\Theta(e)) = (\phi_p)((\pi(e), \theta_p(e)))$
  \item $(\phi_p)(\Theta(e)) =  [\pi(e), p] \cdot e_p \cdot
  \theta_p(e)$
  \item $(\phi_p)(\Theta(e))=[\pi(e), p] \cdot e_p \cdot
  e_p^{-1} \cdot [p,\pi(e)] \cdot e = e$,
\end{enumerate}

\begin{enumerate}
  \item $(\Theta)(\phi_p(x,g))=(\Theta)( [x,p]\cdot e_p \cdot g)$
  \item $(\Theta)(\phi_p(x,g))= (\pi([x,p]\cdot e_p \cdot
  g),\theta_p([x,p]\cdot e_p \cdot g))$
  \item $(\Theta)(\phi_p(x,g))= (x \cdot g)$.
\end{enumerate}

To finish, it suffices to see that the locally trivial projection
$\pi: Z(M,\infty,v_0) \to\ M$ is that of a principal bundle. This
follows by an inspection of the transition functions as follows. Let
$q$ denote a point in $M$ together with an open set $U_q$ in $M$
such that $U_p \cap U_q$ is non-empty. Let $x$ be an element of $U_p
\cap U_q$. Thus the element $[v_p,x,v_q]$ is defined. Define the
function
$$g_{p,q}: U_p \cap U_q \to\ G(M,\infty)$$ by the formula
$$g_{p,q}(x) = e_p^{-1}[v_p,x,v_q]\cdot e_q.$$

These formulas satisfy the following.
\begin{enumerate}
\item $\pi\circ \phi_p(x,g) = x$,
\item  $\phi_p(\pi(e),\theta_p(g)) \pi\circ \phi_p(x,g) = x$
\item  $\theta_p \circ \phi_p(x,g) = g$
\item  $\theta_p \circ \phi_q(x,g) = g_{p,q}(x) \cdot g$.
\end{enumerate}
The lemma follows.
\end{proof}

Dold gave a sufficient condition which insures that certain
functions are fibrations \cite{Albrecht}.
\begin{thm}
If $B$ is a paracompact Hausdorff space, a continuous map $\pi:E
\to\ B$ is a fibration if and only if the map $\pi$ is locally a
fibration.
\end{thm}

The next technical lemma is as follows.
\begin{lem}\label{lem:free loop fibration}
The projection $\pi: X(M,\infty,v_0) \to\ M$ is locally trivial and
a surjection. Thus the projection $\pi$ is a fibration (since $M$ is
paracompact and Hausdorff ).
\end{lem}


\begin{proof}
Fix a point $v$ in $M$. Since $M$ is a Riemannian manifold, the
exists a convex neighborhood $U$ of $v$ by \ref{lem:convex balls}.
The main step in the proof of this lemma is to show that
$\pi^{-1}(U)$ is homeomorphic to $U \times G(M,\infty)$ via a
homeomorphism which preserves the projections to $U$.

Since $M$ is path-connected, there is a path $\alpha$ from $v$ to
$v_0$ ( where $v_0$ is a fixed point in $M$ as given above ). Cover
the image of $\alpha$ by convex neighborhoods choose a finite
subcover ( as the image of the interval is compact ). Thus there is
a finite collection of open convex sets $U =U_0,U_1, \cdots U_q$
such that

\begin{enumerate}
  \item the image of $\alpha$ is contained in the union $\cup_{0 \leq i \leq n}U_i$,
  \item $v$ is in $U_q$,
  \item $v_0$ is in $U = U_0$
  \item $U_i \cap U_{i+1}$ is non-empty for all $i$.
\end{enumerate}

Choose points $w_i$ in $U_i \cap U_{i+1}$ with the condition that
$w_0 = v_0$ $w_q = v$. Since $w_i$ $w_{i+1}$ are in $U_{i+1}$ for
all $i$, the ordered $q+1$-tuple $(v,v_{q-1},v_{q-2},\cdots, v_0)$
is an element of $Z(M,q+1,v_0)$ thus projects to the equivalence
class $[v,v_{q-1},v_{q-2},\cdots, v_0]$ in $Z(M,\infty,v_0)$.

Next, consider any element $$[v_0,x_s, x_{s-1},\cdots,x_1,v_0] =
[v_0,\vec x,v_0]$$ in $G(M,\infty)$. Notice that if $u$ is any point
in $U$, there is an unique minimal geodesic from $u$ to $v$. Thus
$[u,v,v_{q-1},v_{q-2},\cdots, v_0]$ is a point in $Z(M,\infty,v_0)$
such that $\pi([u,v,v_{q-1},v_{q-2},\cdots, v_0]) = u $. Hence the
map $\lambda: G(M,\infty) \times U \to\ X(M,\infty)$ given by the
formula
$$\lambda([v_0,\vec x, v_0],u) = [u,v]\cdot[v,v_{q-1},v_{q-2},\cdots, v_0]\cdot[v_0,\vec x, v_0]$$
is continuous takes values in $\pi^{-1}(U)$. It follows that
$\lambda$ gives a map
$$\lambda: G(M,\infty) \times U \to\ \pi^{-1}(U).$$

Consider any point $\vec z = [y_t,y_{t-1},\cdots, v_0]$  in
$\pi^{-1}(U)$. Thus $y_t$ is a point in $U$, there is an unique
minimal geodesic from $y_t$ to $u$ $[u,y_t,y_{t-1},\cdots, v_0]$ is
a point of $Z(M,\infty,v_0)$.

Define $\gamma: \pi^{-1}(U) \to\ G(M,\infty) \times U $ by the
formula $$\gamma(\vec z) = ([v,v_{q-1},v_{q-2},\cdots, v_0]^{-1}
\cdot [v,\pi(\vec z)], \pi(\vec z)).$$ Notice that $\gamma$ is
continuous is an inverse for $\lambda$ hence locally triviality
follows.

To show that the map $\pi$ is a surjection, choose a path $\alpha$
from $v_0$ to any given point $v$ in $M$. The above cover for the
image of $\alpha$ gives that there is a point
$[v,v_{n},v_{n-1},\cdots, v_0]$ in $X(M,\infty,v_0)$ which projects
to the point $v$.
\end{proof}

\section{Contractibility of $Z(M,\infty,v_0)$}

The proof of Lemma \ref{lem:convex balls} is given first. The
statements to be proven are as follows where $p$ denotes a point of
$M$.

\begin{enumerate}
  \item There exists an $\epsilon > 0$ such
that there is a convex ball of radius $\epsilon$ containing $p$.
  \item Furthermore the function $\Phi:G(M) \to\ map([0,1],M)$
is continuous.
\end{enumerate}

\begin{proof}
Let $p$ be a point in $M$, with $T_p$ the tangent space at the point
$p$. The image of the exponential map applied to $T_p$ contains a
convex neighborhood of $p$. Thus, there exists an unique geodesic
arc joining any two points in this neighborhood there is a smooth
geodesic arc $f: [0,1] \to\ M $ of minimal arc length between these
points. Part $(1)$ of the lemma follows.

Next, it must be shown that $\Phi$ is continuous. Let $U(K,V)$
denote the open set in the function space $map([0,1],M)$ given by
the functions which carry the compact set $K$ in $[0,1]$ into the
open set $V$ in $M$. Notice that since $K$ is a compact subset of
$[0,1]$,
\begin{enumerate}
  \item $K$ is a union of the path components given by $N$ closed
  intervals $[x_j,y_j]$ for a fixed integer $N$ with
  \item $U(K,V) = U(\amalg_{1 \leq j \leq N}[x_j,y_j] ,V)$
  \item $U(K,V) = \cap_{1 \leq j \leq N}U([x_j,y_j],V)$.
\end{enumerate}

Assume that $U(K,V)$ contains the point $f_{(a,b)}(-)$. Thus the
image $f(K)$ is contained in $V$ is compact. Cover the image $f(K)$
by convex open balls contained in $V$ ( by part $(1)$ of the lemma
as $M$ is Riemannian ) choose a finite subcover $O_1,O_2, \cdots ,
O_t$.


Notice that $U(K,O_i)$ is an open subset of $U(K,V)$. Thus the
finite intersection $$\cap_{1 \leq i \leq t}U(K,O_i)$$ is an open
subset of $U(K,V)$. Since the set $K$ is a disjoint union of $N$
closed intervals $[x_j,y_j]$, there is an equality $U(K,V) =
U(\amalg_{1 \leq j \leq N}[x_j,y_j] ,V) = \cap_{\amalg_{1 \leq j
\leq N}}U([x_j,y_j] ,V)$. Thus $\cap_{1 \leq j \leq N} \cap_{1 \leq
i \leq t}U([x_j,y_j],O_i)$ is an open set in $U(K,V)$ which contains
$f$.

To finish, notice that there is a homeomorphism
$$\alpha: \Phi^{-1}(U([x_j,y_j],O_i))\to\ O_i \times O_i$$
which sends a function $h$ to the pair $(h(x_i), h(x_j))$. Thus
$\Phi^{-1}(U([x_j,y_j],O_i))$ is an open set which contains
$f_{(a,b)}(-)$ the set $\cap_{1 \leq j \leq N} \cap_{1 \leq i \leq
t}\Phi^{-1}(U([x_j,y_j],O_i))$ satisfies the following properties:
\begin{enumerate}
  \item the set is open,
  \item the set contains $f$
  \item $f_{(a,b)}(-)$ is in $\Phi^{-1}(U(K,V)$.
\end{enumerate}
The lemma follows.
\end{proof}

The previous lemma is used to prove
\begin{lem}\label{lem:contractability}
If $M$ is a path-connected Riemannian manifold, then
$Z(M,\infty,v_0)$ is contractible.
\end{lem}

\begin{proof}

Consider the subspace of $Z(M,k,v_0)$ given by
$$D(M,k,v_0) = \{(x_k,x_{k-1}, \cdots, x_1,v_0)| x_i \neq x_{i+1},x_{i-1} \neq x_{i+1}, i < k\}.$$
Next, consider the continuous function $\Phi:G(M) \to\ map([0,1],M)$
of  \ref{lem:convex balls} defined by  $\Phi((a,b)) = f_{(a,b)}(-)$,
the unique geodesic from $a$ to $b$ parameterized by arc length.

Assume that $k \geq 1$ let $\vec x = (x_k,x_{k-1}, \cdots, x_1,v_0)$
denote a point in $D(M,k,v_0)$. Define $$h_k:[0,1]\times D(M,k,v_0)
\to\ Z(M,k,v_0)$$ as follows.
$$h_k(t,\vec x )=( \Phi(x_k,x_{k-1})(t), x_{k-1}, x_{k-2}, \cdots, x_1,v_0).$$
Since $\Phi$ is continuous, it follows that $h_k$ is continuous.
Furthermore, $h_k(0,\vec x ) = \vec x$ $h_k(1,\vec x ) = ( x_{k-1},
x_{k-1}, x_{k-2}, \cdots, x_1,v_0)$. In case $k = 0$, define $h_0$
to be the identity. Thus there is a homotopy
$$H:[0,1]\times \amalg_{0 \leq k} D(M,k,v_0) \to\ \amalg_{0 \leq k}
Z(M,k,v_0)$$ given by $\amalg_{0 \leq k} h_k$.

Observe that $H$ passes to quotient spaces to give an induced
continuous function $$\tilde H:[0,1] \times Z(M,\infty,v_0) \to\
Z(M,\infty,v_0)$$ which is continuous when restricted to  $\amalg_{0
\leq k} D(M,k,v_0)$. The homotopy $\tilde H$ gives the property that
$Z(M,\infty,v_0)$ is contractible. The lemma follows.
\end{proof}

%

\section{Maps to free loop spaces}

The purpose of this section is to exhibit the map
$$\Theta:X(M,\infty)\to\ \Lambda(M)$$ by patching together ``small geodesics".

Recall that $v_0$ is a fixed point of $M$ that $Z(M,k)$ is the
subspace of $M^{k+1}$ given by $(k+1)$-tuples $(x_k,x_{k-1}, \cdots
, x_0)$ which satisfy the property that there is an unique minimal
geodesic from $x_i$ to $x_{i+1}$ for all  $1 \leq i < k$ while
$X(M,k)$ denotes the subspace of $Z(M,k)$ with $x_0 = x_k$.

Consider the minimal geodesic from  $x_i$ to $x_{i+1}$ which is
specified by the path $\sigma: [0,1] \to\ M$ given by the smooth
function $f_{(x_i,x_{i+1})}(-)$. Thus given a point $\vec z =
(x_k,x_{k-1}, \cdots , x_0)$ in $X(M,k)$, there are paths
$$\sigma_j: [0,1] \to\ M$$ with $1 \leq j \leq k$ specified by
$\sigma_j(t) = f_{(x_i,x_{i+1})}(t)$. Next, consider the sum of the
arc-lengths of these paths $$ \lambda_1 + \lambda_2+ \cdots +
\lambda_k = L.$$ The paths $\sigma_j$ are glued together to give a
piecewise smooth closed curve as a function from $[0,1]$ to $M$
which is parameterized by arc-length as follows.

Let $\delta_j = (\lambda_1 + \lambda_2+ \cdots + \lambda_j)/L$.
Define $$\sigma_{j,L}: [\delta_{j-1},\delta_j] \to\ M$$ by the
formula
$$\sigma_{j,L}(t) = \sigma_j(x)$$ for $$x = (tL - \delta_{j-1}))/(||tL -
\delta_{j-1})||)$$ and $t$ in $[\delta_{j-1},\delta_j]$. These
functions ``glue" together to give $$\sigma_{\vec z}:[0,1] \to\ M$$
which is given by $\sigma_{j,L}$ when restricted to
$[\delta_{j-1},\delta_j]$. There is an induced function $$\Theta:
X(M,k) \to\ \Lambda (M)$$ defined by $$B(\vec z) =  \sigma_{\vec
z}.$$

\begin{thm}
There is a  morphism of fibrations
\[
\begin{CD}
G(M,\infty) @>{\Theta}>> \Omega (M)   \\
 @VVV           @VV{}V     \\
 X(M,\infty)  @>{\Theta}>> \Lambda (M) \\
 @VVV           @VV{}V     \\
M @>{1}>>  M
\end{CD}
\]
\end{thm}

\begin{proof}
Notice that  the adjoint of $\Theta:X(M,\infty) \to\ \Lambda M$, the
map $\widetilde \Theta:S^1 \times X(M,\infty) \to\ M$ is continuous.
Since all spaces here are Hausdorff and locally compact, the map
$\Theta$ is continuous. That the diagram commutes is a direct
inspection of the definitions.
\end{proof}

\section{Two problems}

It is natural to ask about the existence of natural representations
of the group $G(M,\infty)$ in either $O(m)$ or $GL(m,\mathbb R)$.
Give a natural procedure for constructing such representations.

For example, give a homomorphism $G(M,\infty) \to O(m)$ which
corresponds to either the stable normal bundle or the stable tangent
bundle of $M$. That is, there is a map
$$\nu: M \to BO(m)$$ which represents the normal bundle of $M$. Then
looping this map gives $$\Omega(\nu): \Omega(M) \to \Omega(BO(m))$$
where $\Omega(BO(m))$ is homotopy equivalent to $O(m)$. Describe a
natural representation $\rho:G(M,\infty) \to GL(m,\mathbb R)$ which
is homotopic to $\Omega(\nu)$.

\section{Acknowledgements}

The second author was partially supported by National Science
Foundation and the Institute for Advanced Study ( in 2006 ).

\bibliographystyle{amsalpha}

\end{document}